\documentclass[12pt]{article}

\usepackage{amsfonts}
\usepackage{graphicx}
\usepackage{amsmath}
\usepackage{amssymb}
\usepackage[numbers,square]{natbib}
\usepackage{url}
\usepackage{multirow}
\usepackage{color}
\usepackage{amsthm}

\newtheorem{thrm}{Theorem}
\newtheorem{crllr}[thrm]{Corollary}
\newtheorem{example}[thrm]{Example}
\newtheorem{lmm}[thrm]{Lemma}
\newtheorem{prpstn}[thrm]{Proposition}

\numberwithin{thrm}{section}
\numberwithin{equation}{section}

\title{Negative dependence and stochastic orderings}
\author{Fraser Daly\footnote{Department of Actuarial Mathematics and Statistics, Heriot-Watt University, Edinburgh EH14 4AS, UK.  E-mail: f.daly@hw.ac.uk;  Tel: +44 (0)131 451 3212; Fax: +44 (0)131 451 3249}}
\date{}

\begin{document}

\maketitle

\noindent{\bf Abstract} 
We explore negative dependence and stochastic orderings, showing that if an integer-valued
random variable $W$ satisfies a certain negative dependence assumption, then $W$ is smaller (in the convex sense) than a Poisson variable of equal mean.  
Such $W$ include those which may be written as a sum of totally negatively dependent indicators.  This
is generalised to other stochastic orderings.  Applications include entropy bounds, Poisson approximation and concentration.  
The proof uses thinning and size-biasing.  We also show how these give a different Poisson approximation result,
which is applied to mixed Poisson distributions.  Analogous results for the binomial
distribution are also presented. 
\vspace{12pt}


\vspace{12pt}


\section{Introduction} \label{sec:intro}

Throughout this work we let $W$ be a non-negative, integer-valued random variable with expectation $\lambda>0$.  We focus our attention here on those $W$ which satisfy a certain negative dependence assumption, which we explicitly state in (\ref{eq:order1}) below as a stochastic ordering between $W+1$ and the size-biased version of $W$.  Random variables satisfying this stochastic ordering occur naturally in many applications.  For example, if we may write $W$ as a sum of negatively related Bernoulli random variables, the assumption (\ref{eq:order1}) is satisfied.  Examples of such sums appear in various urn models and occupancy problems, for example.  Several explicit examples of random variables satisfying our negative dependence assumption are discussed in Section \ref{sec:order}.

We are motivated by the work of Daly et al. \cite{dlu12}, who explore links between Stein's method for probability approximation and stochastic orderings.  In their work, as here, these stochastic orderings often reflect the dependence structure of the underlying random variables.  In particular, \cite{dlu12} shows that the stochastic ordering assumption we make here implies a straightforward upper bound on the total variation distance between $W$ and a Poisson random variable.

In this work (and in particular in Section \ref{sec:order} below), we explore further consequences of our stochastic ordering assumption.  In particular, we will see that our negative dependence assumption leads naturally to bounds on the entropy of $W$, concentration inequalities for $W$ and some further Poisson approximation results which complement and enhance those of \cite{dlu12}.  The bounds we derive on entropy generalise entropy maximisation results of Johnson \cite{j07} and Yu \cite{y09}.  See also \cite{jkm13}.  Such results are useful, for example, in understanding probabilistic limit theorems in an information theoretic context. 

Our proofs will make use of the $s$-convex stochastic orders defined by Lef\`evre and Utev \cite{lu96}, which generalise the usual stochastic and convex orderings. We will also need a lemma of Johnson \cite{j07} which links the operations of size-biasing and thinning.  This is stated as Lemma \ref{lem1} later in this section and will be a key tool in what follows.  Further consequences of this lemma will be explored in Section \ref{sec:mp}, where we consider how these thinning and size-biasing results may be applied to Poisson approximation both with and without making any stochastic ordering assumptions.  In particular, we will explore Poisson approximation for a mixed Poisson random variable using these techniques. 

The results and applications we consider in Sections \ref{sec:order} and \ref{sec:mp} are closely related to the Poisson distribution.  This is natural, since Lemma \ref{lem1} is itself closely related to the Poisson distribution.  We will also explore what can be said in relation to the binomial distribution.  This is done in Section \ref{sec:bin}.  We seek the analogues of many of our other results in this case.  For example, under a somewhat different assumption on the dependence  structure of our random variable $W$ to that used in Section \ref{sec:order}, we find binomial approximation results and some further concentration inequalities and bounds on entropy.

We use the remainder of this section to introduce the notation and ideas common to all the work that follows.  We also state the lemma, due to Johnson \cite{j07}, which forms the key to many of the proofs that follow.

For any $\alpha\in[0,1]$, we define the thinning operator $T_\alpha$ by letting $T_\alpha W=\sum_{i=1}^W\eta_i$, where $\eta_1,\eta_2,\ldots$ are iid Bernoulli 
random variables (independent of $W$) with mean $\alpha$.  

Throughout this note, we will let $Z_\mu\sim\mbox{Po}(\mu)$ have a Poisson distribution with mean $\mu$.  The main object we will study in the work that 
follows is the operator $U_\alpha$, given by
\begin{equation}\label{eq:udefn}
U_\alpha W = T_\alpha W +Z_{(1-\alpha)\lambda}\,,
\end{equation}
where $Z_{(1-\alpha)\lambda}$ is independent of all else.  In what follows, for notational convenience we will write $W_\alpha$ for a random variable equal in 
distribution to $U_\alpha W$ for $\alpha\in[0,1]$.  We note that $W_1$ is equal in distribution to $W$, and that $W_0\sim\mbox{Po}(\lambda)$.  

It is easy to see that for any $\alpha\in[0,1]$ we have $\mathbb{E}[W_\alpha]=\mathbb{E}[W]=\lambda$.
We also note that for any $\alpha,\beta\in[0,1]$, $U_\beta(U_\alpha W)$ is equal in distribution to $U_{\alpha\beta} W$.  Finally, it is useful to note that $U_\alpha$ acts trivially on Poisson 
distributions.  That is, $U_\alpha Z_\lambda$ is equal in distribution to $Z_\lambda$ for any $\lambda\geq0$ and $\alpha\in[0,1]$.  
Further properties of the operators $U_\alpha$, and their link with the M/M/$\infty$ queue, are discussed in \cite{j07}.

In what follows, we will also need to employ size biasing.  For any non-negative, integer-valued random variable $W$ with mean $\lambda>0$, we let $W^\star$ denote a 
random variable with the $W$-size-biased distribution, with mass function given by
\begin{equation}\label{eq:sbdef2}
\mathbb{P}(W^\star=j)=\frac{j\mathbb{P}(W=j)}{\lambda}\,,
\end{equation}
for any $j\in\mathbb{Z}^+=\{0,1\ldots\}$.  Equivalently, we may define $W^\star$ by letting
\begin{equation}\label{eq:sbdef}
\mathbb{E}[Wg(W)]=\lambda\mathbb{E}[g(W^\star)]\,,
\end{equation}
for all functions $g:\mathbb{Z}^+\mapsto\mathbb{R}$ for which the expectation exists.  
In a context similar to that considered here, size biasing appears throughout Stein's method for Poisson approximation: we 
refer the interested reader to \cite{bhj92}, \cite{dlu12}, and references therein.  Note that the work we present here is completely distinct from Stein's 
technique, however.

We define the forward difference operator $\Delta$ and its inverse by writing $\Delta f(j)=f(j+1)-f(j)$ and $\Delta^{-1}f(j)=-\sum_{i=j}^\infty f(i)$ for 
$f:\mathbb{Z}^+\mapsto\mathbb{R}$.  Letting $\Delta^0f(j)=f(j)$, we may then define recursively 
$\Delta^nf(j)=\Delta(\Delta^{n-1}f(j))$ and $\Delta^{-n}f(j)=\Delta^{-1}(\Delta^{-n+1}f(j))$ for any $n\geq1$.

We are now in a position to be able to state the following lemma, which appears as Corollary 4.2 of \cite{j07}.
\begin{lmm}\label{lem1}
With $W_\alpha$ as above and $j\in\mathbb{Z}^+$,
$$
\frac{\partial}{\partial\alpha}\mathbb{P}(W_\alpha=j) = \frac{\lambda}{\alpha}\Delta\left[\mathbb{P}(W_\alpha+1=j)-\mathbb{P}(W_\alpha^\star=j)\right]\,.
$$
\end{lmm}
Lemma \ref{lem1} relates the operations of thinning and size biasing, and will be used in establishing stochastic ordering and Poisson approximation results in Sections \ref{sec:order} and \ref{sec:mp}.  A result analogous to Lemma \ref{lem1} will also be needed for the results established in the binomial case and presented in
Section \ref{sec:bin}.  

\section{Negative dependence and convex orderings} \label{sec:order}

In this section we consider the relationship between negative dependence and stochastic ordering.  We will make use of the $s$-convex orderings, defined by Lef\`evre and Utev 
\cite{lu96} for any integer $s\geq1$.  Letting $X$ and $Y$ be non-negative integer-valued random variables, we write $X\leq_{s-cx}Y$ if 
$\mathbb{E}f(X)\leq\mathbb{E}f(Y)$ for all $f\in\mathcal{F}_s$, where
$$
\mathcal{F}_s = \left\{f:\mathbb{Z}^+\mapsto\mathbb{R} \,|\, \Delta^if(j)\geq0\mbox{ for all }j\in\mathbb{Z}^+\mbox{ and }i=1,\ldots,s\right\}\,.
$$
Note that the case $s=1$ corresponds to the
usual stochastic ordering (often denoted by $X\leq_{st}Y$ in what follows) and the case $s=2$ is the increasing convex ordering, written $X\leq_{icx}Y$.  For future 
use, we recall also the standard result that if $\mathbb{E}X=\mathbb{E}Y$ and $X\leq_{icx}Y$ then $X\leq_{cx}Y$, where this denotes the usual convex ordering 
of such random variables.  The interested reader is referred to \cite{ss07} for an introduction to the subject of stochastic orderings.   

Daly et al. \cite{dlu12} give bounds on the Poisson approximation of $W$ in total variation distance under the assumption that 
\begin{equation}\label{eq:order1}
W^\star\leq_{s-cx}W+1\,,
\end{equation}
for some $s\in\mathbb{N}=\{1,2\ldots\}$, where $W^\star$ is defined by (\ref{eq:sbdef2}).  The main result of this section (Theorem \ref{thm:order}) is that the ordering assumption (\ref{eq:order1}) implies an 
ordering between $W$ and a 
Poisson random variable of the same mean.  This yields as an immediate corollary some bounds on Poisson approximation for $W$ and a concentration inequality for $W$.  
From Theorem \ref{thm:order}, we may also derive an upper bound on the entropy of $W$, and hence generalise results of \cite{j07} and \cite{y09}.

Before proceeding further, we note that the stochastic ordering (\ref{eq:order1}) with $s=1$ is closely related to well-known, often applied concepts of negative dependence.  For example, 
if $W=X_1+\cdots+X_n$ for some (dependent) Bernoulli random variables $X_1,\ldots,X_n$ such that
\begin{equation}\label{eq:tnd}
\mbox{Cov}(f(X_i),g(W-X_i))\leq0\,,
\end{equation}
for each $i$ and all increasing functions $f,g:\mathbb{Z}^+\mapsto\mathbb{R}$ then $W+1\geq_{st}W^\star$.  See \cite{pp02} and \cite{dlu12}, where the property 
(\ref{eq:tnd}) is referred to as total negative dependence.  

Recall that Bernoulli random variables $X_1,\ldots,X_n$ are said to be negatively related if 
\begin{equation}\label{eq:nr}
\mathbb{E}\left[\phi(X_1,\ldots,X_{i-1},X_{i+1},\ldots,X_n)|X_i=1\right]\leq\mathbb{E}\left[\phi(X_1,\ldots,X_{i-1},X_{i+1},\ldots,X_n)\right]\,,
\end{equation}
for each $i$ and all increasing functions $\phi:\{0,1\}^{n-1}\mapsto\mathbb{R}$.

Papadatos and Papathanasiou \cite{pp02} showed that if $X_1,\ldots,X_n$ are negatively related then (\ref{eq:tnd}) holds, and hence the stochastic ordering (\ref{eq:order1}) holds with $s=1$.  There are thus many examples and applications which fit into this framework.  We give some illustrative examples below.  In each of these
examples the random variable $W$ may be written as a sum of negatively related Bernoulli variables, and therefore satisfies $W+1\geq_{st}W^\star$.  The negative relation property 
may be established by a straightforward and natural coupling argument in each case. 
\renewcommand{\labelenumi}{(\roman{enumi})}
\begin{enumerate}
\item If $W=X_1+\cdots+X_n$, where $X_1,\ldots,X_n$ are independent Bernoulli random variables then clearly (\ref{eq:nr}) holds.
\item If $W$ has a hypergeometric distribution then Barbour et al. \cite[Section 6.1]{bhj92} show that $W$ may be written as a sum of negatively related Bernoulli random variables.
\item More generally, if we distribute $m$ balls uniformly into $n$ urns and let $W$ count the number of urns which contain at least $c$ balls, Papadatos and Papathanasiou \cite[Section 4]{pp02} show that $W$ may be written as a sum of negatively related Bernoulli random variables.
\item Suppose we have an urn which initially contains balls of $n$ different colours.  We proceed by P\'olya sampling: on each of $m$ draws we choose a ball uniformly 
from the urn, note its colour and return it to the urn along with an additional ball of the same colour.  Let $X_i$ be the indicator that no ball of colour $i$ was 
seen during these $m$ draws.  Then $X_1,\ldots,X_n$ are negatively related: see \cite[Section 6.3]{bhj92}.  Here $W=X_1+\cdots+X_n$ counts the total number of 
colours not seen during the $m$ draws.
\item Consider the following matrix occupancy problem. Suppose we have an $r\times n$ matrix and in row $k$ we place $s_k$ 1s, their positions being chosen by 
uniform sampling without replacement.  All remaining entries of the matrix are set to 0.  Let $T_i$ count the number of 1s in column $i$ and $X_i=I(T_i\leq m)$, the
indicator that column $i$ contains at most $m$ nonzero entries.  Then $W=X_1+\cdots+X_n$ counts the number of such columns.  Barbour et al. \cite{bhj92} show in Section 6.4 that 
$X_1,\ldots,X_n$ are negatively related.
\item Distribute $n$ points uniformly on the circumference of a circle.  Let $S_1,\ldots,S_n$ be the arc-length distances between adjacent points and 
$X_i=I(S_i<a)$, the indicator that $S_i$ falls below some threshold $a$.  Then \cite[Section 7.1]{bhj92} shows that $X_1,\ldots,X_n$ are negatively related.  Their 
sum $W$ counts the number of small spacings on our circle. 
\item Let $(\sigma_1,\ldots,\sigma_n)$ be a permutation of $\{1,\ldots,n\}$ drawn uniformly from the group of such permutations.  Let $X_i=I(\sigma_i\leq a_i)$ for 
some given $a_1,\ldots,a_n$, and $W=X_1+\cdots+X_n$.  In Section 4.1, \cite{bhj92} shows that $X_1,\ldots,X_n$ are negatively related.
\end{enumerate}
In each of these examples we may apply the results of this section.  For further discussion of these examples, and many others, we refer 
the reader to \cite{bhj92,pp02,ab13}, and references therein.  

We now state the main result of this section, Theorem \ref{thm:order}.  In Theorem \ref{thm:order2} we give a slightly stronger result for the case $s=1$.  The proofs of these theorems are deferred until Section \ref{subsec:orderproof}, before which we consider some applications and corollaries.  Note that throughout what follows we let $\binom{a}{b}=0$ if $b>a$. 

\begin{thrm}\label{thm:order}
Let $W$ be a non-negative, integer-valued random variable with $\mathbb{E}[W]=\lambda>0$ and 
$$
\mathbb{E}\binom{W}{k}\leq\mathbb{E}\binom{Z_\lambda}{k}\,,\hspace{20pt} k=3,\ldots,s\,, 
$$
for some $s\in\mathbb{N}$, where $Z_\lambda\sim\mbox{Po}(\lambda)$.  Let $W^\star$ be defined by (\ref{eq:sbdef2}).  If $W^\star\leq_{s-cx}W+1$ then $W\leq_{(s+1)-cx}Z_\lambda$.
\end{thrm}
\begin{thrm}\label{thm:order2}
Let $W$ be a non-negative, integer-valued random variable with $\mathbb{E}[W]=\lambda>0$.  Let $W^\star$ be defined by (\ref{eq:sbdef2}).  If $W^\star\leq_{st}W+1$ then $W_\alpha \leq_{cx} W_\beta$ for 
$\alpha\geq\beta$.  In particular, $W\leq_{cx}Z_\lambda$, where $Z_\lambda\sim\mbox{Po}(\lambda)$. 
\end{thrm}

\subsection{Applications to bounds on entropy} \label{subsec:entropy}

We use this section to give some applications of our Theorem \ref{thm:order2} to upper bounds for entropy.  The bounds we establish
generalise results of \cite{j07} and \cite{y09}.  See also \cite{jkm13}.  
We define the entropy $H(W)$ of a non-negative, integer-valued random variable $W$ in the usual way, although for 
convenience we take natural logarithms. 
$$
H(W)=-\sum_{i=0}^\infty\mathbb{P}(W=i)\log(\mathbb{P}(W=i))\,.
$$
For the random variables we consider here, results are stated which compare their entropy to that of a Poisson random variable with the same mean.  Although no closed-form expression exists for $H(Z_\lambda)$, there are several bounds on this quantity available in the literature.  For example, there is the well-known bound
$$
H(Z_\lambda)\leq \frac{1}{2}\log\left(2\pi e\left(\lambda+\frac{1}{12} \right)\right)\,.
$$ 

In the results that follow, we will also need the notion of log-concavity for a non-negative, integer-valued random variable.  Recall that such a random
variable $W$ is log-concave if its support is an interval in $\mathbb{Z}^+$, and its mass function forms a log-concave sequence.  That is, 
$$
\mathbb{P}(W=i)^2\geq\mathbb{P}(W=i-1)\mathbb{P}(W=i+1)\,,
$$
for all integers $i\geq1$. 
\begin{crllr}\label{cor:entropy1}
Let $W$ be a non-negative, integer-valued random variable with $\mathbb{E}[W]=\lambda>0$.  Let $Z_\lambda\sim\mbox{Po}(\lambda)$.  If $W+1\geq_{st}W^\star$ then
\begin{equation}\label{eq:ent1}
H(W) \leq H(Z_\lambda)\,.
\end{equation}
\end{crllr}
\begin{proof} 
Since $W\leq_{cx}Z_\lambda$ (by Theorem \ref{thm:order2}) and $Z_\lambda$ is a log-concave random variable, the result follows from Lemma 1 
of \cite{y09}.
\end{proof}
Corollary \ref{cor:entropy1} shows that $Z_\lambda$ maximises the entropy within our class of $W$ with expectation $\lambda$ and such that $W+1\geq_{st}W^\star$.  
We again note that the conclusion of Corollary \ref{cor:entropy1} holds if $W$ may be written as a sum of totally negatively dependent (or negatively related) 
Bernoulli random variables, as in our examples above.  Such maximum entropy results are of importance in understanding probabilistic limit theorems in an information 
theoretic context.  For further discussion of this, we refer the reader to \cite{j07} and references therein.

Corollary \ref{cor:entropy1} generalises Theorem 2.5 of \cite{j07}, which states that (\ref{eq:ent1}) holds under the assumption that $W$ is 
ultra log-concave (of degree $\infty$), denoted ULC($\infty$) in what follows.  Recall that $W$ is ULC($\infty$) if
$$
(j+1)!^2\mathbb{P}(W=j+1)^2 \geq j!(j+2)!\mathbb{P}(W=j)\mathbb{P}(W=j+2),\hspace{15pt}j\geq0\,,
$$
or, equivalently, if
$$
\frac{(j+1)\mathbb{P}(W=j+1)}{\mathbb{P}(W=j)}\,,
$$
is increasing in $j$.  We note that this is equivalent to $W+1\geq_{lr}W^\star$, where `$\geq_{lr}$' denotes the likelihood ratio ordering.  Since this is 
stronger than stochastic ordering \cite[Theorem 1.C.1]{ss07}, our Corollary \ref{cor:entropy1} strengthens Theorem 2.5 of \cite{j07}.  Similarly, Corollary 
\ref{cor:entropy2} below generalises Theorem 3 of \cite{y09}.  See also \cite{jkm13}.

\begin{crllr}\label{cor:entropy2}
Let $W$ be a non-negative, integer-valued random variable such that $\mathbb{E}[W]=\lambda>0$ and $W+1\geq_{st}W^\star$.  Let $Z_\lambda\sim\mbox{Po}(\lambda)$ and 
$X_1,X_2,\ldots,$ be iid non-negative, integer-valued random variables.  Let
$$
\widehat{W}=\sum_{i=1}^WX_i\,,\hspace{20pt}\mbox{ and }\hspace{20pt}\widehat{Z_\lambda}=\sum_{i=1}^{Z_\lambda}X_i\,.
$$
If $\widehat{Z_\lambda}$ is log-concave, then $H(\widehat{W})\leq H(\widehat{Z_\lambda})$.
\end{crllr} 
\begin{proof} 
Combine our Theorem \ref{thm:order2} with Theorem 1 of \cite{y09}.
\end{proof}
For discussion of the log-concavity assumption used in this result (including some sufficient conditions for $\widehat{Z_\lambda}$ to be log-concave), we refer 
the reader to Section 3 of \cite{y09} and Section 5 of \cite{jkm13}.  In particular, \cite[Theorem 4]{y09} shows that if $X_1$ is log-concave and
$$
\lambda\mathbb{P}(X_1=1)^2\geq2\mathbb{P}(X_1=2)\,,
$$
then $\widehat{Z_\lambda}$ is log-concave.

Johnson \cite{j07} goes further than establishing that the Poisson distribution maximises entropy within the class of ULC($\infty$) random variables of mean $\lambda$.
In his Theorem 5.1 he shows that for such $W$ the entropy of $W_\alpha$ is a decreasing and concave function of $\alpha$.  Using our stochastic ordering arguments we may also 
generalise this result, and show it applies to $W$ satisfying $W+1\geq_{st}W^\star$.  This is done in Theorem \ref{thm:entropy}.

\begin{thrm}\label{thm:entropy}
Let $W$ be a non-negative, integer-valued random variable satisfying $W+1\geq_{st}W^\star$, where $W^\star$ is defined by (\ref{eq:sbdef2}).  Then
\begin{equation}\label{eq:entthm}
\frac{\partial}{\partial\alpha}H(W_\alpha)\leq0\hspace{20pt}\mbox{ and }\hspace{20pt}\frac{\partial^2}{\partial\alpha^2}H(W_\alpha)\leq0\,,
\end{equation}
with equality if and only if $W$ has a Poisson distribution.
\end{thrm}
\begin{proof}
Our proof uses many of the same components of that of Theorem 5.1 of \cite{j07}, but replacing the arguments based on ultra log-concavity
with stochastic ordering results.  Following \cite{j07} we decompose the entropy as
$$
H(W_\alpha) = \Lambda(W_\alpha)-D(W_\alpha\lVert Z_\lambda)\,,
$$
where $Z_\lambda\sim\mbox{Po}(\lambda)$,
\begin{eqnarray*}
\Lambda(W_\alpha)&=&-\sum_{j=0}^\infty\mathbb{P}(W_\alpha=j)\log\left(\mathbb{P}(Z_\lambda=j)\right)\,,\\
D(W_\alpha\lVert Z_\lambda)&=&\sum_{j=0}^\infty\mathbb{P}(W_\alpha=j)\log\left(\frac{\mathbb{P}(W_\alpha=j)}{\mathbb{P}(Z_\lambda=j)}\right)\,.
\end{eqnarray*}
Note that $D$ here is the relative entropy.  Lemmas 5.2 and 5.5 of \cite{j07} give us immediately that
$$
\frac{\partial}{\partial\alpha}D(W_\alpha\lVert Z_\lambda)\geq0\hspace{20pt}\mbox{ and }\hspace{20pt}\frac{\partial^2}{\partial\alpha^2}D(W_\alpha\lVert Z_\lambda)\geq0\,,
$$
since for $W$ such that $W+1\geq_{st}W^\star$ we have $\mbox{Var}(W)\leq\mathbb{E}[W]$.

To prove (\ref{eq:entthm}), it remains only to show that $\Lambda(W_\alpha)$ is a decreasing and concave function of $\alpha$.  By equation (15) of \cite{j07}
we have that
\begin{equation}\label{eq:l1}
\frac{\partial}{\partial\alpha}\Lambda(W_\alpha)=\frac{\lambda}{\alpha}\left\{\mathbb{E}\log(W_\alpha^\star)-\mathbb{E}\log(W_\alpha+1)\right\}\,.
\end{equation}
We will see in Section \ref{subsec:orderproof} that $W$ such that $W+1\geq_{st}W^\star$ satisfy the ordering $W_\alpha+1\geq_{st}W_\alpha^\star$ for each $\alpha\in[0,1]$.  Since 
$\log(\cdot)$ is an increasing function, it immediately follows from (\ref{eq:l1}) that $\Lambda(W_\alpha)$ is a decreasing function of $\alpha$.

Similarly, from Lemma 5.3 of \cite{j07} we have that
$$
\frac{\partial^2}{\partial\alpha^2}\Lambda(W_\alpha)=\frac{\lambda^2}{\alpha^2}\left\{\mathbb{E}f(W_\alpha^\star)-\mathbb{E}f(W_\alpha+1)\right\}\,, 
$$
where
$$
f(j)=\frac{j-1}{\lambda}\log\left(\frac{j}{j+1}\right)-\log\left(\frac{j+1}{j}\right)\,.
$$
Since $f(\cdot)$ is an increasing function, we see that $\Lambda(W_\alpha)$ is a concave function of $\alpha$, completing the proof of (\ref{eq:entthm}).

The fact that equality holds in (\ref{eq:entthm}) if and only if $W$ has a Poisson distribution is shown in the same way as the corresponding statement in Theorem
5.1 of \cite{j07}. 
\end{proof}
We have already discussed several examples in which the results of this section may be directly applied.  We conclude with an example where we may use our results even without the negative dependence assumption (\ref{eq:order1}): the lightbulb process.  This model was introduced by Rao et al. \cite{rrz07}, and is motivated by the pharmaceutical problem of a dermal patch designed to target $n$ receptors.  Each receptor is in one of two states.  On each day $r=1,\ldots,n$, the patch causes $r$ uniformly selected receptors to switch state.  

This process has also been studied, for example, by Goldstein and Zhang \cite{gz11}, and Goldstein and Xia \cite{gx12}.  See also references therein.  It is more often described in terms of lightbulbs being switched on and off, with $r$ of the $n$ lightbulbs chosen uniformly to have their state switched at day $r$, for $r=1,\ldots,n$.  For concreteness, we assume that all $n$ bulbs are switched off at the start of the process.  The random variable of interest is $W=W(n)$, the number of bulbs switched on after day $n$. We consider here the problem of bounding the entropy of $W$.

Goldstein and Zhang \cite{gz11} show that (at least for $n$ even) $W^\star\leq_{st}W+2$, but this is not enough to apply our results directly.  Instead, we use the fact, shown by Goldstein and Xia \cite{gx12}, that $W$ is asymptotically distributed as a clubbed binomial distribution.  If we let $X\sim\mbox{Bin}(n-1,1/2)$ have a binomial distribution, then we define the clubbed binomial random variable $Y=Y_m$ by writing
\begin{equation*}
 \mathbb{P}(Y_m=j) = \left\{
\begin{array}{ll}
\mathbb{P}(X=j-1)+\mathbb{P}(X=j) & \mbox{$m$ and $j$ have the same parity,} \\
0 & \mbox{otherwise.} \\
\end{array} \right.  \end{equation*}
That is, the clubbed binomial is formed by combining the mass of the binomial distribution at adjacent integers, so that it is supported on the lattice of non-negative integers with the same parity as $m$.  We note that the support of these clubbed binomial distributions is appropriate to the problem at hand since, as shown by Rao et al. \cite{rrz07}, if $n\equiv 0$ (mod 4) or $n\equiv 3$ (mod 4) then $W$ is supported on the set of even integers at most $n$.  Otherwise, the support of $W$ is the set of odd integers at most $n$.  It what follows, we always choose $m$ in the definition of $Y$ appropriately for the $n$ under consideration. 

We begin with the straightforward observation that $H(Y)\leq H(X)$.  This follows immediately from the definition of $Y$.  Since our binomial distribution $X$ satisfies $X^\star\leq_{st}X+1$, it follows from Corollary \ref{cor:entropy1} that
\begin{equation}\label{eq:light1}
H(Y)\leq H(Z_{(n-1)/2})\,,
\end{equation}
where $Z_\lambda\sim\mbox{Po}(\lambda)$ as usual.  Hence, using (\ref{eq:light1}), we have that
$$
H(W) \leq H(Z_{(n-1)/2}) + \left| H(W)-H(Y) \right|\,.
$$ 
This last term may be bounded using Theorem 17.3.3 of \cite{ct}, which states that if $W$ and $Y$ are random variables supported on a subset of $\mathbb{Z^+}$ of size $k$ and 
$$
\sum_{j\in\mathbb{Z}^+}\left| \mathbb{P}(W=j)-\mathbb{P}(Y=j) \right| \leq \beta \leq \frac{1}{2}\,,
$$
then
$$
\left| H(W)-H(Y) \right| \leq -\beta\log\left(\frac{\beta}{k}\right)\,.
$$
We may apply this result here with the choice
\begin{equation}\label{eq:light2}
\beta = 5.47\sqrt{n}\exp\left(\frac{-(n+1)}{3}\right)\,,
\end{equation}
by Theorem 3.1 of \cite{gx12}, noting that $\beta\leq1/2$ for $n\geq10$.  Since both $W$ and $Y$ are supported on either the even or odd integers up to $n$, we may take $k=(n/2)+1$.  Hence we obtain the following.
\begin{crllr}
Let $W=W(n)$ be the number of bulbs switched on at the terminal time of the lightbulb process.  Then, with $n\geq10$ and $\beta$ given by (\ref{eq:light2}), 
$$ 
H(W) \leq H(Z_{(n-1)/2}) - \beta\log\left(\frac{2\beta}{n+2}\right)\,.
$$
\end{crllr}

\subsection{Applications to Poisson approximation}\label{subsec:approx}

For further applications of our Theorems \ref{thm:order} and \ref{thm:order2} we turn to some Poisson approximation results.  For use here and in 
Section \ref{sec:mp}, we define the probability metrics we will use.  In this framework, we are inspired by the 
recent work of R\"ollin and Ross \cite{rr12}.  For $1\leq p<\infty$ and $f:\mathbb{Z}^+\mapsto\mathbb{R}$ we let
$$
\lVert f \rVert_p = \left(\sum_{j=0}^\infty |f(j)|^p\right)^{1/p}\,,
$$
and we let $\lVert f \rVert_\infty=\sup_j|f(j)|$.  For distribution functions $F$ and $G$, we then define the distances
$$
d_{n,p}(F,G) = \lVert\Delta^nF-\Delta^nG\rVert_p\,,
$$
for $1\leq p\leq\infty$ and $n\in\mathbb{Z}$.  Many commonly-used probability metrics fit into this framework.  For example,
\begin{itemize}
\item the total variation distance: $d_{TV}(F,G)=\frac{1}{2}d_{1,1}(F,G)$.
\item the Kolmogorov distance: $d_K(F,G)=d_{0,\infty}(F,G)$.
\item the Wasserstein distance: $d_W(F,G)=d_{0,1}(F,G)$.
\item the stop-loss distance: $d_{SL}(F,G)=d_{-1,\infty}(F,G)$.
\end{itemize}
Note also that $d_{1,\infty}$ is a metric useful in proving local limit theorems.

To provide an illustration of the type of Poisson approximation result which may be obtained in our framework, in this section we will consider approximation in the metrics $d_{-k,\infty}$ for 
$k\geq-1$.  The results of Section \ref{sec:mp} below will use some of the other probability metrics we have defined.  In the work of this section, we are motivated by the techniques and results of \cite{dlu02}.

\begin{crllr}\label{cor:order}
Let $W$ be as in Theorem \ref{thm:order}.  If $W$ has distribution function $F$ and $Z_\lambda$ has distribution function $G_\lambda$ then
$$
d_{-k,\infty}(F,G_\lambda) \leq 2^{(s-k-1)_+}\mathbb{E}\left[\binom{Z_\lambda+s+1}{s+1}-\binom{W+s+1}{s+1}\right]\,,
$$
for $k=-1,\ldots,s+1$.
\end{crllr}
\begin{proof} 
The result follows from Theorem \ref{thm:order} and an argument analogous to that for Corollary 3.14 of \cite{dlu02}.
\end{proof}
If we take $s=1$ in Corollary \ref{cor:order} we obtain that, for any $W$ with $W+1\geq_{st}W^\star$,
\begin{equation}\label{eq:hyp1}
d_{-k,\infty}(F,G_\lambda) \leq 2^{(-k)_+-1}\left(\lambda-\mbox{Var}(W)\right)\,,
\end{equation}
for $k\in\{-1,0,1,2\}$ (and hence including bounds on the stop-loss, Kolmogorov and local limit distances).  This applies in particular if $W$ may be written as 
a sum of totally negatively dependent Bernoulli random variables, as in the examples discussed previously.

We conclude this section with a short example to illustrate this result.

\begin{example}\label{eg:h1}
\emph{Suppose we distribute $m$ balls uniformly among $N>1$ urns, where each urn has the capacity for up to one ball.  Let $W$ count the number of the first $n$ urns that are occupied.  Then $W$ has a hypergeometric distribution with mean $\lambda=mn/N$ and variance
$$
\frac{mn}{N}\left(\frac{N-n}{N-1}\right)\left(1-\frac{m}{N}\right)\,.
$$
As noted earlier, $W$ may be written as a sum of negatively related Bernoulli random variables and so satisfies $W+1\geq_{st}W^\star$.  The bound (\ref{eq:hyp1}) then gives
$$
d_{-k,\infty}(F,G_\lambda)\leq2^{(-k)_+-1}\frac{mn}{N}\left(\frac{(m+n)N-mn-N}{N(N-1)}\right)\,,
$$
for $k\in\{-1,0,1,2\}$, where $F$ is the distribution function of $W$ and $G_\lambda$ the distribution function of a Poisson random variable with mean $\lambda$.}

\emph{We note that upper bounds of a better order may be available.  For example, let $k=0$ (so that we consider the Kolmogorov distance) and suppose that $m$ and $n$ are both of order $O(N)$.  Then our upper bound is also of order $O(N)$, but an upper bound of better order $O(1)$ is available from Theorem 6.A of Barbour et al. \cite{bhj92}.  However, our results have the advantage of dealing simultaneously with a range of probability metrics.}
\end{example}

\subsection{Applications to concentration inequalities}

In this section we note that the convex ordering of Theorem \ref{thm:order2} implies a concentration inequality for $W$.
\begin{crllr}\label{cor:conc}
Let $W$ be a non-negative, integer-valued random variable such that $\mathbb{E}W=\lambda$ and $W+1\geq_{st}W^\star$.  Let $t>0$.  Then
\begin{eqnarray*}
\mathbb{P}(W\geq\lambda+t) &\leq& e^t\left(1+\frac{t}{\lambda}\right)^{-(t+\lambda)}\,,\\
\mathbb{P}(W\leq\lambda-t) &\leq& e^{-t}\left(1-\frac{t}{\lambda}\right)^{t-\lambda}\,,
\end{eqnarray*}
where the latter bound applies if $t<\lambda$.
\end{crllr}
\begin{proof}
To prove the first inequality, let $\theta>0$ and note that (by a standard argument using Markov's inequality)
$$
\mathbb{P}(W-\lambda\geq t)\leq \exp\left\{-\theta(t+\lambda)\right\}\mathbb{E}e^{\theta W}\,.
$$
Now, for $\theta>0$, the function $e^{\theta x}$ is convex in $x$ and hence we apply Theorem \ref{thm:order2} to note that
$$
\mathbb{E}e^{\theta W} \leq \exp\left\{\lambda\left(e^\theta-1\right)\right\}\,.
$$
We then minimize the resulting bound over $\theta$.  The proof of the second inequality is similar.
\end{proof}
These inequalities have also been found in recent work by Arratia and Baxendale \cite[Theorem 4.2]{ab13}, who show they perform well compared to other such concentration inequalities which are available.

\subsection{Proofs of Theorems \ref{thm:order} and \ref{thm:order2}} \label{subsec:orderproof}

We now give the proofs of Theorems \ref{thm:order} and \ref{thm:order2}.  We begin with some properties of the $s$-convex orderings.  Our Lemmas \ref{lem:order1}--\ref{lem:order3}
will make use of results established by Denuit and Lef\`evre \cite{dl97} and Denuit et al. \cite{dlu99}.  In particular, we will need closure of the $s$-convex orderings under operations such 
as convolution and taking mixtures.  

\begin{lmm}\label{lem:order1}
Let $X$ and $Y$ be non-negative, integer-valued random variables.  If $X\leq_{s-cx}Y$ for some $s\in\mathbb{N}$ then $T_\alpha X\leq_{s-cx}T_\alpha Y$ for all
$\alpha\in[0,1]$.
\end{lmm}
\begin{proof}
To see this, use Property 4.6 of \cite{dlu99} and a proof analogous to that of Theorem 8.A.13 of \cite{ss07}.
\end{proof}
\begin{lmm}\label{lem:order2}
Let $W$ be a non-negative, integer-valued random variable with positive mean.  If we have $W^\star\leq_{s-cx}W+1$ for some $s\in\mathbb{N}$ then 
$(T_\alpha W)^\star\leq_{s-cx}T_\alpha W+1$ for all $\alpha\in[0,1]$.
\end{lmm}
\begin{proof}
Using Lemma \ref{lem:order1} and the closure of the $s$-convex orders under convolution \cite[Proposition 3.7]{dl97} , we have that
$$
T_\alpha W+1 \geq_{s-cx} T_\alpha(W^\star-1)+1 = T_\alpha(VW)+1\,,
$$
where the operator $V$ is defined by $VW=W^\star-1$.  Since the operators $T_\alpha$ and $V$ commute (as can be easily checked) we obtain
$$
T_\alpha W+1 \geq_{s-cx} V(T_\alpha W)+1 = (T_\alpha W)^\star\,,
$$
as required.
\end{proof}
\begin{lmm}\label{lem:order3}
Let $X_1$ and $X_2$ be independent non-negative, integer-valued random variables with positive mean.  If $X_1^\star\leq_{s-cx}X_1+1$ and $X_2^\star\leq_{s-cx}X_2+1$ 
for some $s\in\mathbb{N}$ then $(X_1+X_2)^\star\leq_{s-cx}X_1+X_2+1$.
\end{lmm}
\begin{proof} 
We firstly note that $(X_1+X_2)^\star=X_1+X_2-X_I+X_I^\star$, where the random index $I\in\{1,2\}$ is chosen independently of all else and such that
$$
\mathbb{P}(I=1)=1-\mathbb{P}(I=2)=\frac{\mathbb{E}X_1}{\mathbb{E}X_1+\mathbb{E}X_2}\,.
$$
See \cite[Corollary 2.1]{cgs12}, for example.  Conditioning on the event that $I=1$, we have
$$
(X_1+X_2)^\star = X_1^\star+X_2 \leq_{s-cx}X_1+X_2+1\,,
$$
by assumption and using Proposition 3.7 of \cite{dl97}.  An analogous argument holds if we condition instead on the event that $I=2$.  To complete the proof we remove 
the conditioning using Proposition 3.7 of \cite{dl97}.
\end{proof}
We are now in a position to give the proof of Theorem \ref{thm:order}.  Noting that Poisson random variables trivially satisfy the ordering (\ref{eq:order1}) for 
all $s\in\mathbb{N}$, Lemmas \ref{lem:order2} 
and \ref{lem:order3} may be combined to give us that for $W$ satisfying the assumptions of our theorem, 
\begin{equation}\label{eq:order2}
W_\alpha^\star\leq_{s-cx}W_\alpha+1\,,
\end{equation}
for all $\alpha\in[0,1]$. 

Now, following Lef\`evre and Utev \cite{lu96}, we let $h_0(X,j)=\mathbb{P}(X=j)$ for any non-negative, integer-valued random variable $X$ and $j\in\mathbb{Z}^+$.  We define 
$h_k(X,j)$ for $k\geq1$ by letting 
\begin{equation}\label{eq:hdef}
h_k(X,j)=-\Delta^{-1}h_{k-1}(X,j)=\sum_{i=j}^\infty h_{k-1}(X,i)=\mathbb{E}\binom{X-j+k-1}{k-1}\,.
\end{equation}
By Proposition 2.5 of \cite{lu96}, to prove that $W\leq_{(s+1)-cx}Z_\lambda$, we need to show that
\begin{equation}\label{eq:lu1}
 \mathbb{E}\binom{W}{k}\leq\mathbb{E}\binom{Z_\lambda}{k}\,,\hspace{20pt}k=1,\ldots,s\,,
\end{equation}
and that
\begin{equation}\label{eq:lu2}
 h_{s+1}(W,j)\leq h_{s+1}(Z_\lambda,j)\,,\hspace{20pt}j\geq s+1\,.
\end{equation}
Beginning with (\ref{eq:lu1}), the inequality with $k=1$ is trivial, since $\mathbb{E}Z_\lambda=\lambda$.  In the case $k=2$, it is straightforward to show, 
using (\ref{eq:sbdef}), that if
$\mathbb{E}W^\star\leq\mathbb{E}W+1$ (which holds by the assumption that $W^\star\leq_{s-cx}W+1$) then $\mathbb{E}\binom{W}{2}\leq\mathbb{E}\binom{Z_\lambda}{2}$.  
The remaining cases, $k=3,\ldots,s$ are covered explicitly in the statement of Theorem \ref{thm:order}.

It remains only to establish (\ref{eq:lu2}).  Lemma \ref{lem1} gives us that
$$
\frac{\partial}{\partial\alpha}h_0(W_\alpha,j) = \frac{\lambda}{\alpha}\Delta\left[h_0(W_\alpha+1,j)-h_0(W_\alpha^\star,j)\right]\,.
$$  
Applying $\Delta^{-(s+1)}$ to each side of this equation (and interchanging summation and differentiation) we obtain
$$
-\frac{\partial}{\partial\alpha}h_{s+1}(W_\alpha,j) = \frac{\lambda}{\alpha}\left[h_s(W_\alpha+1,j)-h_s(W_\alpha^\star,j)\right]\,.
$$
By the stochastic ordering (\ref{eq:order2}) and Proposition 2.5 of \cite{lu96}, $h_s(W_\alpha+1,j)\geq h_s(W_\alpha^\star,j)$ for all $\alpha$ and $j$.  Hence, letting 
$j\in\mathbb{Z}^+$,
\begin{eqnarray*}
0&\leq&\int_0^1\frac{\lambda}{\alpha}\left[h_s(W_\alpha+1,j)-h_s(W_\alpha^\star,j)\right]\,d\alpha\\ &=& -\int_0^1\frac{\partial}{\partial\alpha}h_{s+1}(W_\alpha,j)\,d\alpha\\&=&h_{s+1}(Z_\lambda,j)-h_{s+1}(W,j)\,,
\end{eqnarray*}
as required, since $W_1$ is equal in distribution to $W$ and $W_0\sim\mbox{Po}(\lambda)$. This establishes our Theorem  \ref{thm:order}.

The proof of Theorem \ref{thm:order2} is exactly as for Theorem \ref{thm:order} above (with $s=1$), except for a change in the limits of integration.

\subsection{Remarks on some related results}

We conclude Section \ref{sec:order} by noting some results related to Theorem \ref{thm:order2}.  Before stating these, we need a
definition.  We recall that random variables $\{X_i:i\in\Gamma\}$ are negatively associated if
$$
\mbox{Cov}\left(f(X_i,i\in\Gamma_1),g(X_i,i\in\Gamma_2)\right) \leq0\,,
$$
for all increasing functions $f$ and $g$ and all $\Gamma_1,\Gamma_2\subseteq\Gamma$ with $\Gamma_1\cap\Gamma_2=\emptyset$.  Negative association is closely related to other concepts
of negative dependence we have used.  For example, note that negatively associated indicator random variables are negatively related, and hence sums of negatively associated indicator variables satisfy our stochastic ordering assumption (\ref{eq:order1}) with $s=1$.

Shao \cite{s00} shows that if $X_1,\ldots,X_n$ are negatively associated and if the random variables $X^\dagger_1,\ldots,X^\dagger_n$ are 
independent with each of the $X^\dagger_i$ having the same marginal distribution as $X_i$ then
\begin{equation}\label{eq:shao}
X_1+\cdots+X_n\leq_{cx}X^\dagger_1+\cdots+X^\dagger_n\,.
\end{equation}
In the case where $X_1,\ldots,X_n$ are indicator random variables, the stochastic comparison (\ref{eq:shao}) with the sum of independent random variables is 
a stronger result than our Theorem \ref{thm:order2}, in which the comparison is with a Poisson variable.  We note, though, that our results apply in a more general
negative dependence setting, and that we obtain results for the more general $s$-convex orderings (as in our Theorem \ref{thm:order}).

Results analogous to (\ref{eq:shao}) are also available in a positive dependence setting.  Recall that random variables $\{X_i:i\in\Gamma\}$ are associated if
$$
\mbox{Cov}\left(f(X_i,i\in\Gamma),g(X_i,i\in\Gamma)\right) \geq0\,,
$$
for all increasing functions $f$ and $g$.  Denuit et al. \cite{ddr01} show that if $X_i,\ldots,X_n$ are associated then
$$
X_1+\cdots+X_n\geq_{cx}X^\dagger_1+\cdots+X^\dagger_n\,.
$$
In the course of this work, we have been unable to find results in a positive dependence setting, such as for sums of associated random variables.

\section{Further results in Poisson approximation} \label{sec:mp}

In Section \ref{subsec:approx} we saw how our negative dependence assumption leads to bounds in Poisson approximation for our random variable $W$.  These bounds were established using the convex ordering given by Theorem \ref{thm:order}, which was itself proved using Lemma \ref{lem1}.  We use this section to give another application of thinning and size biasing (via our Lemma \ref{lem1}) to Poisson approximation. 

We state a bound in Lemma 
\ref{lem2} below which will be applied (in Section \ref{subsec:wass}) to give some general results in Poisson approximation which do not need any assumptions of stochastic ordering.  
We will note, however, the refinements and simplifications available in these results if we introduce the same stochastic ordering assumptions which we used in Section \ref{sec:order}.  

In Section \ref{subsec:mp} we will apply Lemma \ref{lem2} to the problem of Poisson approximation of the mixed Poisson distribution. 

Our results will be stated in terms of the distances $d_{n,p}$ defined in Section \ref{subsec:approx}.

\begin{lmm} \label{lem2}
Let $W$ be a non-negative, integer-valued random variable with distribution function $F$ and $\mathbb{E}[W]=\lambda>0$.  Let $G_\lambda$ be the distribution function 
of $Z_\lambda\sim\mbox{Po}(\lambda)$.  Then for $1\leq p\leq\infty$ and $n\in\mathbb{Z}$
$$
d_{n,p}(F,G_\lambda) \leq \lambda\int_0^1\frac{1}{\alpha}d_{n+1,p}(F_\alpha^{(1)},F_\alpha^\star)\,d\alpha\,,
$$ 
where $F_\alpha^{(1)}$ is the distribution function of $W_\alpha+1$ and $F_\alpha^\star$ is the distribution function of $W_\alpha^\star$. 
\end{lmm}
\begin{proof} 
Let $F_\alpha$ be the distribution function of $W_\alpha$.  We use the definition of $d_{n,p}$ and note that $F=F_1$ and
$G_\lambda=F_0$ to obtain
\begin{eqnarray*}
d_{n,p}(F,G_\lambda) &=& \left\lVert\Delta^n\int_0^1\frac{\partial}{\partial\alpha}F_\alpha\,d\alpha \right\rVert_p\\
&\leq& \int_0^1 \left\lVert\Delta^n\frac{\partial}{\partial\alpha}F_\alpha\right\rVert_p\,d\alpha\\
&=&\lambda\int_0^1\frac{1}{\alpha}\left\lVert\Delta^{n+1}F_\alpha^{(1)}-\Delta^{n+1}F_\alpha^\star\right\rVert_p\,d\alpha\,,
\end{eqnarray*}
where the inequality follows from Minkowski's integral inequality \cite[Appendix A]{x} and the final line uses Lemma \ref{lem1}.
\end{proof}

\subsection{Poisson approximation using thinning and size biasing}\label{subsec:wass}

The main result of this section is Theorem \ref{thm:pois} below.  This contains some Poisson approximation results derived from Lemma \ref{lem2} and also shows how these results may be combined with the same stochastic ordering assumption employed in Section \ref{sec:order}. 

To ease the notational burden on this section we will write $d_{n,p}(X,Y)$ to mean $d_{n,p}(F,G)$ if $X$ and $Y$ are random variables with distribution functions $F$ and $G$, respectively. 

\begin{thrm}\label{thm:pois}
Let $W$ be a non-negative, integer-valued random variable with $\mathbb{E}[W]=\lambda>0$ and let $W^\star$ be defined by (\ref{eq:sbdef2}).  Let $Z_\lambda\sim\mbox{Po}(\lambda)$.  
\begin{enumerate}
\renewcommand{\theenumi}{\alph{enumi}}
\renewcommand{\labelenumi}{(\alph{enumi})}
\item\label{th32_1} For $s\in\mathbb{Z}^+$
\begin{eqnarray}
\label{eq:31}d_{-s,1}(W,Z_\lambda)&\leq&\frac{\lambda}{1+s}d_{1-s,1}(W,W^\star-1)\,,\\
\label{eq:32}d_{-s,\infty}(W,Z_\lambda)&\leq&\frac{\lambda}{s}d_{1-s,\infty}(W,W^\star-1)\,,
\end{eqnarray}
where this last inequality applies if $s\not=0$.
\item\label{th32_2} If, in addition, $W+1\geq_{s-cx}W^\star$ then
\begin{eqnarray}
\label{eq:33}d_{-s,1}(W,Z_\lambda)&\leq&\frac{1}{1+s}\mathbb{E}\left[\lambda\binom{W+s}{s}-W\binom{W+s-1}{s}\right]\,,\\
\nonumber d_{-k,\infty}(W,Z_\lambda)&\leq&\frac{2^{(s-k-1)_+}}{k}\mathbb{E}\left[\lambda\binom{W+s}{s}-W\binom{W+s-1}{s}\right]\,,\\\label{eq:34}
\end{eqnarray}
for $k=1,\ldots,s+1$.
\end{enumerate}
\end{thrm}
\begin{proof}
We begin by using Corollary 2.1 of \cite{cgs12}, and the fact that $Z_{\lambda}^\star$ is equal in distribution to $Z_\lambda+1$ for all $\lambda$, to note that 
\begin{equation}\label{eq:sb}
W_\alpha^\star=(T_\alpha W+Z_{(1-\alpha)\lambda})^\star=I_\alpha(T_\alpha W)^\star+(1-I_\alpha)(T_\alpha W+1)+Z_{(1-\alpha)\lambda}\,,
\end{equation}
where $I_\alpha$ is independent of all else and $\mathbb{P}(I_\alpha=1)=\alpha=1-\mathbb{P}(I_\alpha=0)$.

Using the functions $h_s(X,j)$ defined by (\ref{eq:hdef}) for any non-negative random variable $X$ and $j\in\mathbb{Z}^+$, we have that for $s\in\mathbb{Z}^+$,
$$
d_{-s,p}(W_\alpha+1,W_\alpha^\star)=\lVert h_{s+1}(W_\alpha+1,\cdot)-h_{s+1}(W_\alpha^\star,\cdot)\rVert_p\,.
$$ 
With $W_\alpha^\star$ given by (\ref{eq:sb}), we can condition on $I_\alpha$ and $Z_{(1-\alpha)\lambda}$ to get that
$$
d_{-s,p}(W_\alpha+1,W_\alpha^\star)\leq\alpha d_{-s,p}(T_\alpha W,V(T_\alpha W))\,,
$$
where the operator $V$ is such that $VX=X^\star-1$ for any non-negative random variable $X$.  Since the operators $V$ and $T_\alpha$ commute for any $0\leq\alpha\leq1$, we have that
\begin{equation}\label{eq:dist}
d_{-s,p}(W_\alpha+1,W_\alpha^\star)\leq\alpha d_{-s,p}(T_\alpha W,T_\alpha (W^\star-1))\,.
\end{equation}
Recalling that $T_\alpha W=\sum_{j=1}^W\eta_j$, where $\eta_1,\eta_2,\ldots$ are iid Bernoulli variables with mean $\alpha$,  in the case that $p\in\{1,\infty\}$ we may bound this latter distance using an argument analogous to that of Proposition 4.2 of Denuit and Van Bellegem \cite{dv01} to get that
\begin{eqnarray}
\label{eq:31_1}d_{-s,1}(T_\alpha W,T_\alpha (W^\star-1))&\leq&\alpha^{s+1}d_{-s,1}(W,W^\star-1) \,,\\
\label{eq:31_2}d_{-s,\infty}(T_\alpha W,T_\alpha (W^\star-1))&\leq&\alpha^{s}d_{-s,\infty}(W,W^\star-1) \,,
\end{eqnarray}
We may now complete the proof of the first part of the theorem.  Combining Lemma \ref{lem2} with (\ref{eq:dist}) and (\ref{eq:31_1}) we have that
$$
d_{-s,1}(W,Z_\lambda)\leq\lambda d_{1-s,1}(W,W^\star-1)\int_0^1\alpha^s\,d\alpha=\frac{\lambda}{1+s}d_{1-s,1}(W,W^\star-1)\,.
$$ 
Similarly, using (\ref{eq:31_2}) in place of (\ref{eq:31_1}), we have that if $s\not=0$
$$
d_{-s,\infty}(W,Z_\lambda)\leq\lambda d_{1-s,\infty}(W,W^\star-1)\int_0^1\alpha^{s-1}\,d\alpha=\frac{\lambda}{s}d_{1-s,\infty}(W,W^\star-1)\,.
$$
This completes the proof of part (\ref{th32_1}).

For part (\ref{th32_2}), we note that if $W+1\geq_{s-cx}W^\star$ then 
$$
d_{1-s,1}(W,W^\star-1)=\mathbb{E}\left[\binom{W+s}{s}-\binom{W^\star-1+s}{s}\right]\,.
$$  
Combining this with (\ref{eq:31}) and (\ref{eq:sbdef}) gives us (\ref{eq:33}).

Now let $k\in\{1,\ldots,s+1\}$.  Corollary 3.14 of \cite{dlu02} gives us that if $W+1\geq_{s-cx}W^\star$ then 
$$
d_{1-k,\infty}(W,W^\star-1)\leq2^{(s-k-1)_+}\mathbb{E}\left[\binom{W+s}{s}-\binom{W^\star-1+s}{s}\right]\,.
$$   
We obtain (\ref{eq:34}) when we combine this with (\ref{eq:32}) and (\ref{eq:sbdef}).
\end{proof}
To illustrate this result, we consider two examples.
\begin{example}\label{eg:h2}
\emph{Firstly, we return to the setting of Example \ref{eg:h1} and let $W$ have a hypergeometric distribution (with notation as in Example \ref{eg:h1}).  Then, letting
$$
\epsilon=\frac{mn}{N}\left(\frac{(m+n)N-mn-N}{N(N-1)}\right)\,,
$$
and recalling that $W+1\geq_{st}W^\star$ in this case, Theorem \ref{thm:pois}(\ref{th32_2}) gives $d_{-1,1}(W,Z_\lambda)\leq\epsilon/2$, $d_{-2,\infty}(W,Z_\lambda)\leq\epsilon/2$ and $d_{-1,\infty}(W,Z_\lambda)\leq\epsilon$, where this latter metric is the stop-loss distance.}
\end{example}
\begin{example}\label{eg:p1}
\emph{We consider now the P\'olya distribution, which has found applications in the study of epidemics, genetics and communications.  Suppose we have an urn containing $N$ balls, of which $r$ are red and $N-r$ are black.  At each step, we draw a ball, note its colour and return it to the urn together with $c\geq1$ additional balls of the same colour.  Repeat this for a total of $m$ draws, and let $W$ count the number of red balls chosen in these $m$ draws.  Then $W$ has a P\'olya distribution with mean $\lambda=mr/N$ and variance given by
$$
\sigma^2=\frac{mr(N+cm)(N-r)}{N^2(N+c)}\,.
$$
We use Theorem \ref{thm:pois}(\ref{th32_1}) to give a bound on the Wasserstein distance $d_W(W,Z_\lambda)=d_{0,1}(W,Z_\lambda)$.  From that result, we have that
$$
d_W(W,Z_\lambda)\leq2\lambda d_{TV}(W,W^\star-1)\leq2\lambda\left\{d_{TV}(W,W^\star)+d_{TV}(W,W+1)\right\}\,,
$$
where the final bound is the triangle inequality.  From inequalities (5) and (20) of \cite{d11}, respectively, we have that $d_{TV}(W,W^\star)\leq\sigma/2\lambda$ and
$$
d_{TV}(W,W+1)\leq\frac{1}{2\lambda(N-r)\sqrt{N+c}}\left\{\sqrt{mr(N-r)(N+cm)}+m\sqrt{cr(N-r)}\right\}\,.
$$
Combining these inequalities, we have the following explicit bound:
\begin{multline}\label{eq:polya}
d_W(W,Z_\lambda)\leq\sqrt{\frac{mr(N+cm)(N-r)}{N^2(N+c)}}\\+\frac{1}{(N-r)\sqrt{N+c}}\left\{\sqrt{mr(N-r)(N+cm)}+m\sqrt{cr(N-r)}\right\}\,.
\end{multline}
Some further discussion of the P\'olya distribution, and the bound (\ref{eq:polya}), is given in Example \ref{eg:p2} below.}
\end{example}
We note that the results of Theorem \ref{thm:pois} are not the only way in which our stochastic ordering assumption can be used to get a Poisson approximation result based on Lemma \ref{lem2}.  For example, consider the Wasserstein distance $d_W=d_{0,1}$ and total variation distance $d_{TV}=\frac{1}{2}d_{1,1}$.  An argument analogous to that used to obtain (\ref{eq:dist}) gives us that
$$
d_{TV}(W_\alpha+1,W_\alpha^\star)\leq\alpha d_{TV}(T_\alpha W,T_\alpha(W^\star-1))\,.
$$  
Combining this with Lemma \ref{lem2} (in the case $n=0$) we have that
$$
d_W(W,Z_\lambda)\leq2\lambda\int_0^1d_{TV}(T_\alpha W,T_\alpha(W^\star-1))\,d\alpha\,.
$$
If we assume that $W+1\geq_{st}W^\star$, we may use Theorem 7 of \cite{rp03} to obtain
$$
d_W(W,Z_\lambda)\leq2\lambda\int_0^1\alpha\mathbb{E}\left[W+1-W^\star\right]\,d\alpha=\lambda-\mbox{Var}(W)\,.
$$
In this case, however, better bounds are available by combining Proposition 2 of \cite{dlu12} with Theorem 1.1 of \cite{bx06}.  We thus obtain
$$
d_W(\mathcal{L}(W),\mathcal{L}(Z_\lambda))\leq\left(1\wedge\frac{1.15}{\sqrt{\lambda}}\right)\left(\lambda-\mbox{Var}(W)\right)\,.
$$

\subsection{Poisson approximation for mixed Poisson distributions}\label{subsec:mp}

In this section we apply Lemma \ref{lem2} to the case where $W\sim\mbox{Po}(\xi)$ has a mixed Poisson distribution with positive mixture distribution $\xi$ and 
$\mathbb{E}[\xi]=\lambda$.  We begin by showing that in this case, $W_\alpha$ also has a mixed Poisson distribution.  Note that we will not make any assumptions of stochastic or convex ordering in this section.

\begin{lmm}\label{lem3}
If $W\sim\mbox{Po}(\xi)$ then $W_\alpha\sim\mbox{Po}(\alpha\xi+(1-\alpha)\lambda)$ for all $\alpha\in[0,1]$.
\end{lmm}
\begin{proof}
Elementary computations show that for $j\in\mathbb{Z}^+$
$$
\mathbb{P}(T_\alpha W=j) = \sum_{i=0}^\infty\binom{i}{j}\alpha^j(1-\alpha)^{i-j}\mathbb{P}(W=i) = \frac{1}{j!}\mathbb{E}\left[e^{-\alpha\xi}(\alpha\xi)^j\right]\,,
$$
so that $T_\alpha W\sim\mbox{Po}(\alpha\xi)$.  Since $W_\alpha$ is the convolution of $T_\alpha W$  and an independent Poisson random variable, the result follows.   
\end{proof}
Now, let us write $\xi_{(\alpha)}=\alpha\xi+(1-\alpha)\lambda$ and
$$
g_\alpha(j) = \frac{\exp\{-{\xi_{(\alpha)}}\}\xi_{(\alpha)}^{j-1}}{(j-1)!}\,.
$$
Since
$$
\mathbb{P}(W_\alpha+1=j)-\mathbb{P}(W_\alpha^\star=j)=\mathbb{E}\left[\left(1-\frac{\xi_{(\alpha)}}{\lambda}\right)g_\alpha(j)\right]\,,
$$
Lemmas \ref{lem2} and \ref{lem3} give us that
\begin{eqnarray*}
d_{n,p}(F,G_\lambda) &\leq& \int_0^1\frac{1}{\alpha}\left\lVert\Delta^n\mathbb{E}\left[(\xi_{(\alpha)}-\lambda)g_\alpha\right]\right\rVert_p\,d\alpha\\
&\leq& \int_0^1\frac{1}{\alpha}\mathbb{E}\left[|\xi_{(\alpha)}-\lambda|\left\lVert\Delta^ng_{\alpha}\right\rVert_p\right]\,d\alpha\,,
\end{eqnarray*}
where we have again used Minkowski's integral inequality.

For the remainder of this section we focus only on the case $n\geq0$ and $p=1$.  In this case, straightforward calculations using Lemma 3.4 of \cite{rr12} give us that
$$
\left\lVert\Delta^ng_{\alpha}\right\rVert_1 \leq \xi_{(\alpha)}^{-n/2}\,,
$$
and hence
\begin{equation}\label{eq:mp1}
d_{n,1}(F,G_\lambda) \leq \int_0^1\mathbb{E}\left[|\xi-\lambda|\left(\alpha\xi+(1-\alpha)\lambda\right)^{-n/2}\right]\,d\alpha\,.
\end{equation}
If we assume that the expectation in (\ref{eq:mp1}) exists for all $\alpha\in[0,1]$ then we may interchange the order of integration to obtain the following result.

\begin{thrm}\label{thm:mp}
Let $W\sim\mbox{Po}(\xi)$ for some positive random variable $\xi$ with $\mathbb{E}[\xi]=\lambda$.  Let $F$ be the distribution function of $W$ and $G_\lambda$ be 
the distribution function of $Z_\lambda\sim\mbox{Po}(\lambda)$.  Suppose that
$$
\mathbb{E}\left[|\xi-\lambda|\left(\alpha\xi+(1-\alpha)\lambda\right)^{-n/2}\right]<\infty\,,
$$
for some $n\in\mathbb{Z}^+$ and all $\alpha\in[0,1]$.  Then if $n\not=2$,
$$
d_{n,1}(F,G_\lambda) \leq \left|\frac{2}{n-2}\right|\mathbb{E}\left|\xi^{(2-n)/2}-\lambda^{(2-n)/2}\right|\,,
$$
while if $n=2$, $d_{2,1}(F,G_\lambda)\leq\mathbb{E}\left|\log\left(\frac{\xi}{\lambda}\right)\right|$.
\end{thrm}
We illustrate this result by returning to the setting of Example \ref{eg:p1}, the P\'olya distribution.
\begin{example}\label{eg:p2}
\emph{Let $W$ have a P\'olya distribution, as described in Example \ref{eg:p1}.  We show how Theorem \ref{thm:mp} may be used to give a bound on the Wasserstein distance $d_W(W,Z_\lambda)=d_{0,1}(W,Z_\lambda)$ between $W$ and a Poisson distribution of the same mean.  To do this, we follow \cite{aa08} and construct $W$ as the mixed binomial distribution $\mbox{Bin}(m,\xi)$, where $\xi$ has a beta distribution with density function
$$
g(t)=B(\alpha,\beta)^{-1}t^{\alpha-1}(1-t)^{\beta-1}\,,
$$
for $t\in(0,1)$, where $B(\cdot,\cdot)$ is the beta function, $\alpha=r/c$ and $\beta=(N-r)/c$.}

\emph{Letting $Y\sim\mbox{Po}(m\xi)$ have a mixed Poisson distribution, we may condition on $\xi$ to obtain the bound $d_W(W,Y)\leq1.15\sqrt{m}\mathbb{E}\left[\xi^{3/2}\right]$ from equation (1.8) of \cite{bx06}.  Our Theorem \ref{thm:mp} gives $d_W(Y,Z_{\lambda})\leq m\mathbb{E}|\xi-p|$, where $p=\mathbb{E}\xi=r/N$.  The triangle inequality and H\"older's inequality then give
\begin{eqnarray}
\nonumber d_W(W,Z_\lambda)&\leq&1.15\sqrt{m}\mathbb{E}\left[\xi^{3/2}\right]+m\mathbb{E}|\xi-p|\\
\nonumber&\leq&1.15\sqrt{m}\left(\mathbb{E}\xi^2\right)^{3/4}+m\sqrt{\mbox{Var}(\xi)}\\
\label{eq:polya2}&=&1.15\sqrt{m}\left(\frac{r(r+c)}{N(N+c)}\right)^{3/4}+m\sqrt{\frac{cr(N-r)}{N^2(N+c)}}\,.
\end{eqnarray} 
Asymptotically, this bound behaves similarly to that derived in Example \ref{eg:p1} above.  For example, if $m$ is of order $O(N)$ and $c$ and $r$ are both of order $O(1)$, then each of the bounds (\ref{eq:polya}) and (\ref{eq:polya2}) are of order $O(1)$.  However, numerical studies suggest that in practice (\ref{eq:polya2}) performs better than (\ref{eq:polya}).}
\end{example} 
In the case of the total variation distance $d_{TV}=\frac{1}{2}d_{1,1}$, Theorem \ref{thm:mp} gives the following.

\begin{crllr}\label{cor:mp}
Let $F$ and $G_\lambda$ be as in Theorem \ref{thm:mp}.  For any $\epsilon\in[0,1/2]$
$$
d_{TV}(F,G_\lambda) \leq \frac{\left(\mathbb{E|\xi-\lambda|}\right)^{1/2+\epsilon}}{\lambda^\epsilon}\,.
$$
\end{crllr}
\begin{proof} 
From Theorem \ref{thm:mp} we have that $d_{TV}(F,G_\lambda)\leq\mathbb{E}|\sqrt{\xi}-\sqrt{\lambda}|$.  This expectation may be bounded 
using Lemma 1 of \cite{r03} and H\"older's inequality to give the required result.  
\end{proof}
We note however, that bounds superior to that given by Corollary \ref{cor:mp} may be available elsewhere.  For example, consider the case where $W$ has a negative binomial distribution.  That is, assume that $\xi$ has a gamma distribution with density function
$$
g(t) = \frac{1}{\Gamma(\beta)}\left(\frac{1-q}{q}\right)^\beta t^{\beta-1}\exp\left\{-t\left(\frac{1-q}{q}\right)\right\}\,,
$$ 
for $t>0$, for some $\beta\in(0,\infty)$ and $q\in(0,1)$.  Note that $\lambda=\beta q(1-q)^{-1}$, $\mbox{Var}(\xi)=\beta q^2(1-q)^{-2}$ and
$$
\mathbb{E}|\xi-\lambda| = \frac{2q\beta^\beta e^{-\beta}}{(1-q)\Gamma(\beta)}\leq\frac{q}{1-q}\sqrt{\frac{2\beta}{\pi}}\,,
$$
where this inequality uses a slight generalisation of Proposition A.2.9 of \cite{bhj92} whose proof is straightforward.  Thus, evaluating the bound of 
Corollary \ref{cor:mp}, and in particular with the choices $\epsilon=0$ and $\epsilon=1/2$, we obtain that in the negative binomial case
\begin{equation}\label{eq:nb1}
d_{TV}(F,G_\lambda)\leq\sqrt{\frac{q}{1-q}}\min\left\{\sqrt{\frac{2}{\pi}},\left(\frac{2\beta}{\pi}\right)^{1/4}\right\}\,.
\end{equation}
For comparison, Roos \cite{r03} obtains the bound
\begin{equation}\label{eq:nb2}
d_{TV}(F,G_\lambda)\leq\beta\left(\frac{q}{1-q}\right)^2\min\left\{\frac{3(1-q)}{4e\beta q},1\right\}\,,
\end{equation}
and shows that it is superior to many others available in the literature.  Note that, regardless of the value of $\beta$, the bound (\ref{eq:nb1}) is of order $O(\sqrt{q})$ while (\ref{eq:nb2}) has order at least as good as $O(q)$. 

\section{The binomial case}\label{sec:bin}

The results that we have stated in previous sections (based on Lemma \ref{lem1}) are closely related to the Poisson distribution, since Lemma \ref{lem1} is 
itself closely related to the Poisson distribution. In 
this section we turn our attention to results in the binomial case.  We consider results analogous to those in Sections \ref{sec:order} and \ref{sec:mp}.  In doing 
this, 
we will use a Markov chain constructed by Yu \cite{y08} and used in proving an upper bound on entropy.  

We begin with some useful definitions.  Throughout this section let $W$ be a non-negative, integer-valued random variable supported on $\{0,1,\ldots,n\}$, for 
some integer $n>0$, and with mean $\lambda=nr>0$.  We will let $Z\sim\mbox{Bin}(n,r)$, a binomial random variable with the same support and mean as $W$.  

We recall that a random variable $X$ supported on $\{0,1,\ldots,n\}$ is ultra log-concave of degree $n$, denoted ULC($n$) in the sequel, if
$$
\frac{\mathbb{P}(X=i+1)^2}{\binom{n}{i+1}^2}\geq\frac{\mathbb{P}(X=i)}{\binom{n}{i}}\frac{\mathbb{P}(X=i+2)}{\binom{n}{i+2}}\,,
$$
for $0\leq i\leq n-2$.  We refer the reader to \cite{p00} for further discussion of this property.  We note here that the ULC($n$) property is intended to capture 
negative dependence, in a similar way to the ULC($\infty$) property and the other negative dependence assumptions we have discussed in Section \ref{sec:order}.

For those $W$ which are ULC($n$), Yu \cite[Theorem 1]{y08} proves that $H(W)\leq H(Z)$.  This is an analogue of Theorem 2.5 of \cite{j07}, which we generalised in 
our Corollary \ref{cor:entropy1}.  The proof of Yu's result employs a Markov chain $\{X_t:t\in\mathbb{Z}^+\}$, whose construction we now outline.  Further details 
and discussion are provided by \cite{y08}. 

We let $X_0$ have the same distribution as $W$.  The random variable $X_{t}$ (for $t\geq1$) is given by \begin{equation}\label{eq:xdef}
X_{t}=H_n(X_{t-1}+\eta_{t-1})\,,
\end{equation} 
where $\eta_0,\eta_i,\ldots$ are iid Bernoulli random variables with mean $r$ and the operator $H_n$ is such that for a non-negative, integer-valued random variable 
$X$ supported on $\{0,1,\ldots,n\}$
$$
\mathbb{P}(H_nX=i)=\frac{(n-i)}{n}\mathbb{P}(X=i)+\frac{(i+1)}{n}\mathbb{P}(X=i+1)\,,
$$ 
for $0\leq i\leq n-1$.  The operator $H_n$ is referred to as hypergeometric thinning, since, conditional on $X$, $H_nX$ has a hypergeometric distribution.  This is the
analogue of the (binomial) thinning operator $T_\alpha$ defined in Section \ref{sec:intro}.  Recall that, conditional on $X$, $T_\alpha X$ has a binomial distribution.  

In proving his entropy bound, Yu \cite{y08} uses the random variables $\{X_t:t\in\mathbb{Z}^+\}$ in a role analogous to that of the random variables 
$\{W_\alpha:0\leq\alpha\leq1\}$ in the corresponding bound for the Poisson case \cite[Theorem 2.5]{j07}.  We use the remainder of this section to examine how the 
techniques we have developed in our previous work may be carried over into this binomial setting.  We begin with the analogue of Lemma \ref{lem1}.

Writing $p_t(i)=\mathbb{P}(X_t=i)$, Yu \cite{y08} shows that for $t\geq0$
\begin{equation}\label{eq:mc}
p_{t+1}(i)=\frac{(n+1-i)(sp_t(i)+rp_t(i-1))+(i+1)(sp_t(i+1)+rp_t(i))}{n+1}\,,
\end{equation}
where $s=1-r$.  We note that $X_t$ is supported on $\{0,1,\ldots,n\}$ and has expectation $nr$ for each $t\in\mathbb{Z}^+$.  The key property of this Markov chain 
is that as $t\rightarrow\infty$, $X_t$ converges in distribution to the binomial distribution $\mbox{Bin}(n,r)$.

Now, given a random variable $W$ supported on $\{0,1,\ldots,n\}$, define the random variable $W^+$ by 
$$
\mathbb{P}(W^+=j)=\frac{n+1-j}{n(1-r)}\mathbb{P}(W+1=j)\,,
$$
for $1\leq i\leq n$.  Straightforward manipulations of (\ref{eq:mc}) then allow us to see the following result, analogous to our Lemma \ref{lem1} for the Poisson case.
\begin{lmm}\label{lem:bin1}
Let $W$ be a random variable supported on $\{0,1,\ldots,n\}$ with mean $nr>0$.  Then for $t\in\mathbb{Z}^+$ and $0\leq j\leq n$
$$
\mathbb{P}(X_t=j)-\mathbb{P}(X_{t+1}=j) = \frac{nr(1-r)}{n+1}\Delta\left[\mathbb{P}(X_t^+=j)-\mathbb{P}(X_t^\star=j)\right]\,.
$$
\end{lmm}

\subsection{Convex ordering and ULC($n$)}

We use the next part of this section to explore stochastic ordering properties similar to those considered previously in Section \ref{sec:order}.  We will make use 
of ultra log-concavity, and will assume that $W$ is ULC($n$).  For such $W$ we have that $W^+\geq_{st}W^\star$ and that 
$X_t$ is ULC($n$) for all $t\in\mathbb{Z}$.  See \cite[Lemma 3]{y08}.  Combining these facts we immediately see that if $W$ is ULC($n$) then $X_t^+\geq_{st}X_t^\star$ 
for all $t\in\mathbb{Z}^+$.  We may then derive the following result, which plays the role of Theorem \ref{thm:order2} in the binomial case.
\begin{thrm}\label{thm:binorder}
Let $W$ be ULC($n$) with support $\{0,1,\ldots,n\}$ and mean $nr>0$.  Let $Z\sim\mbox{Bin}(n,r)$ and $X_t$ be given by (\ref{eq:xdef}).  Then $X_t\leq_{cx}X_u$ for all $t\leq u$.  
In particular, $W\leq_{cx}Z$.
\end{thrm}
\begin{proof} 
We use the ideas and notation of the proof of Theorem \ref{thm:order}.  As in the proof of that result, Proposition 2.5 of \cite{lu96} 
gives us that we need only show that $h_2(X_t,j)\leq h_2(X_{t+1},j)$ for each $t\in\mathbb{Z}^+$ and $0\leq j\leq n$.  The first statement in the theorem follows 
easily from this, and the final statement by taking $t=0$ and $u\rightarrow\infty$ in the first.  

As noted before, for $W$ a ULC($n$) random variable, we have that $X_t^+\geq_{st}X_t^\star$ for each $t\in\mathbb{Z}^+$.  Hence $h_1(X_t^+,j)\geq h_1(X_t^\star,j)$ 
for all $t\in\mathbb{Z}^+$ and $0\leq j\leq n$.

Now, by Lemma \ref{lem:bin1} we have that
$$
h_0(X_t,j)-h_0(X_{t+1},j)=\frac{nr(1-r)}{n+1}\Delta\left[h_0(X_t^+,j)-h_0(X_t^\star,j)\right]\,,
$$ 
for each $t\in\mathbb{Z}^+$ and $0\leq j\leq n$.  Applying $\Delta^{-2}$ to this, we have that
$$
\frac{nr(1-r)}{n+1}\left[h_1(X_t^+,j)-h_1(X_t^\star,j)\right]=h_2(X_{t+1},j)-h_2(X_t,j)\geq0,
$$
as required.  
\end{proof}
From Theorem \ref{thm:binorder}, we may immediately recover the main result of Yu \cite{y08}, his Theorem 1, which we state in Corollary \ref{cor:binentropy} below.
\begin{crllr}\label{cor:binentropy}
Let $W$ be ULC($n$) with support $\{0,1,\ldots,n\}$ and mean $nr>0$.  Let $Z\sim\mbox{Bin}(n,r)$.  Then
$$
H(W) \leq H(Z)\,.
$$
\end{crllr}
\begin{proof}
Since $W\leq_{cx}Z$ (by Theorem \ref{thm:binorder}) and $Z$ is a log-concave random variable, this follows immediately from Lemma 1 of 
\cite{y09}.
\end{proof}
We also have the following, the analogue of Corollary \ref{cor:entropy2}.
\begin{crllr}\label{cor:binentropy2}
Let $W$ be ULC($n$) with support $\{0,1,\ldots,n\}$ and mean $nr>0$.  Let $Z\sim\mbox{Bin}(n,r)$ and 
$Y_1,Y_2,\ldots,$ be iid non-negative, integer-valued random variables.  Let
$$
\widehat{W}=\sum_{i=1}^WY_i\,,\hspace{20pt}\mbox{ and }\hspace{20pt}\widehat{Z}=\sum_{i=1}^{Z}Y_i\,.
$$
If $\widehat{Z}$ is log-concave, then $H(\widehat{W})\leq H(\widehat{Z})$.
\end{crllr}
\begin{proof}
Combine our Theorem \ref{thm:binorder} with Theorem 1 of \cite{y09}.
\end{proof}
Note that Corollary \ref{cor:binentropy2} generalises Theorem 2 of \cite{y09}, since a sum of $n$ independent Bernoulli random variables is ULC($n$).

We conclude this subsection by observing that we may also obtain concentration inequalities and binomial approximation results as corollaries of our Theorem
\ref{thm:binorder}, as in the Poisson case of Section \ref{sec:order}.  The proofs of these results are analogous to their Poisson counterparts in Section \ref{sec:order}.
\begin{crllr}
Let $W$ be ULC($n$) with support $\{0,1,\ldots,n\}$ and mean $\lambda=nr>0$.  Let $t>0$.
\begin{eqnarray*}
\mathbb{P}(W\geq\lambda+t) &\leq& \left[\frac{(1-r)(\lambda+t)}{(1-r)\lambda-rt}\right]^{-(t+\lambda)}\left[1-r+\frac{r(1-r)(\lambda+t)}{(1-r)\lambda-rt}\right]^n\,,\\ 
\mathbb{P}(W\leq\lambda-t) &\leq& \left[\frac{(1-r)(\lambda-t)}{(1-r)\lambda+rt}\right]^{t-\lambda}\left[1-r+\frac{r(1-r)(\lambda-t)}{(1-r)\lambda+rt}\right]^n\,,\\
\end{eqnarray*}
where the last inequality applies if $t<\lambda$.
\end{crllr}
\begin{crllr}
Let $W$ be ULC($n$) with support $\{0,1,\ldots,n\}$ and mean $nr>0$.  Let $Z\sim\mbox{Bin}(n,r)$.  Then if $W$ has distribution function $F$ and
$Z$ has distribution function $G$,
$$
d_{-k,\infty}(F,G)\leq2^{(-k)_+-1}\left\{nr(1-r)-\mbox{Var}(W)\right\}\,,
$$
for $k\in\{-1,0,1,2\}$. 
\end{crllr}

\subsection{Other results in the binomial case}

In Section \ref{sec:mp} we used our Lemma \ref{lem1} directly to provide a Poisson approximation result, Lemma \ref{lem2}.  Similarly, we have the following.
\begin{prpstn}\label{lem:bin2}
Let $W$ be a random variable supported on $\{0,1,\ldots,n\}$ with mean $nr>0$.  Let $F_t$ be the distribution function of $X_t$, for $t\in\mathbb{Z}^+$.  Then 
for $1\leq p\leq\infty$ and $n\in\mathbb{Z}$
$$
d_{n,p}(F_0,F_t)\leq\frac{nr(1-r)}{n+1}\sum_{u=0}^{t-1}d_{n+1,p}(F_u^+,F_u^\star)\,,
$$
where $F_u^+$ is the distribution function of $X_u^+$ and $F_u^\star$ is the distribution function of $X_u^\star$.
\end{prpstn}
\begin{proof}
From the definition of $d_{n,p}$ we have that
\begin{eqnarray*}
d_{n,p}(F_0,F_t)&=&\left\lVert\Delta^{n}\sum_{u=0}^{t-1}[F_u-F_{u+1}]\right\rVert_p\\
&=&\frac{nr(1-r)}{n+1}\left\lVert\Delta^{n+1}\sum_{u=0}^{t-1}[F_u^+-F_u^\star]\right\rVert_p\\
&\leq&\frac{nr(1-r)}{n+1}\sum_{u=0}^{t-1}\left\lVert\Delta^{n+1}F_u^+-\Delta^{n+1}F_u^\star\right\rVert_p\,,
\end{eqnarray*}
where the second line uses Lemma \ref{lem:bin1} and the inequality uses Minkowski's integral inequality.
\end{proof}
It is worth noting, however, that we do not have a result analogous to Lemma \ref{lem3} here.  That is, suppose that $X_0=W\sim\mbox{Bin}(n,\xi)$ for some random 
variable $\xi$ supported on $[0,1]$, so that
$$
\mathbb{P}(W=i)=\binom{n}{i}\mathbb{E}[\xi^i(1-\xi)^{n-i}]\,.
$$
Then $X_1$ does not, in general, have a mixed binomial distribution.  
In the Poisson case, the preservation of Poisson mixtures under the operators $U_\alpha$ ($0\leq\alpha\leq1$), as given by Lemma \ref{lem3}, allowed us to easily and 
explicitly find a bound on the distance between a mixed Poisson random variable and a Poisson random variable with the same mean.  However, no such property holds in 
the binomial case we are considering here.

\subsection*{Acknowledgements}
The author gratefully acknowledges useful and interesting discussions with Oliver Johnson and Sergey Utev.  Thanks are also due to the Heilbronn Institute for Mathematical Research at the University of Bristol, where part of this work was completed, and to an anonymous referee, whose suggestions improved the quality and presentation of the work.

\end{document}